\documentclass{amsart}
\usepackage{graphicx}
\usepackage{newlattice}

\DeclareMathOperator{\At}{At}
\newcommand{\tl}{\textup{l}}

\newcommand{\SK}[1]{\msf{K}_{#1}}

\newcommand{\slo}{<_\textup{lr}}

\theoremstyle{plain}
\newtheorem{theorem}{Theorem}
\newtheorem{lemma}[theorem]{Lemma}

\newtheorem{corollary}[theorem]{Corollary}

\theoremstyle{definition}

\begin{document}

\title{Atom-generated planar lattices}

\author{G. Gr\"{a}tzer} 
\email[G. Gr\"atzer]{gratzer@me.com}
\urladdr[G. Gr\"atzer]{http://server.maths.umanitoba.ca/homepages/gratzer/}

\date{Submitted Jan. 14, 2020; revised Sept. 3, 2020}

\begin{abstract}
In this note, we discuss planar lattices 
generated by their atoms.
We prove that if $L$ is a planar lattice 
generated by $n$ atoms, 
then both the left and the right boundaries of $L$ 
have at most $n+1$ elements.

On the other hand, $L$ can be arbitrarily large. 
For every $k > 1$, 
we construct a planar lattice $L$
generated by $4$ atoms such that 
$L$ has more than $k$ elements.
\end{abstract}

\maketitle

\section{Introduction}\label{S:intro}

In this note, we deal with 
\emph{atom-generated planar lattices} 
(AGP lattices, for short),
that is, planar lattices (finite, by definition) 
generated by their sets of atoms; 
see Figure~\ref{F:4} for some small examples
and Figures \ref{F:4genstep1}--\ref{F:4geninduct} 
for some larger ones.

\begin{figure}[htb]
\centerline{\includegraphics{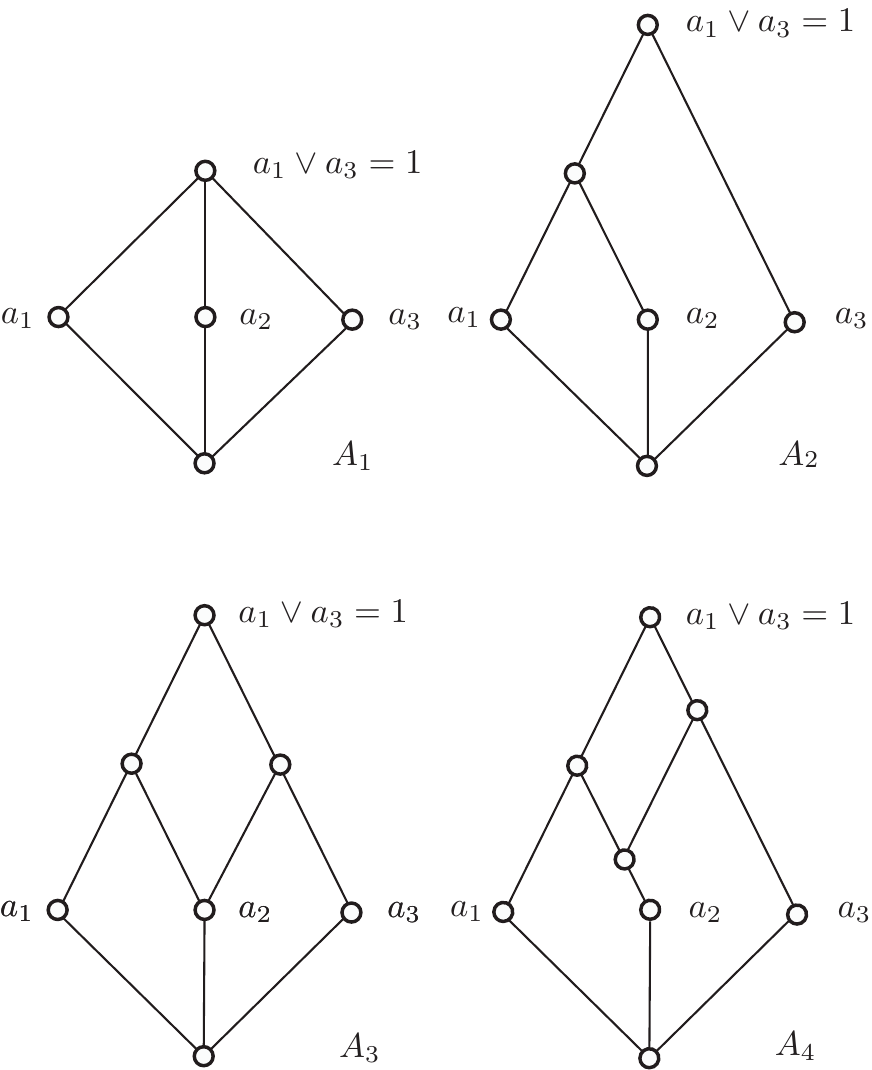}}
\caption{$3$-generated AGP lattices}\label{F:4}
\end{figure}

We have two results.

\begin{theorem}\label{T:boundary}
Let $L$ be a planar lattice generated its atoms. 
If $L$ has $n$ atoms, 
then both the left and the right boundaries of $L$ 
have at most $n+1$ elements.
\end{theorem}

\begin{theorem}\label{T:size}
For every pair of integers $k,n$ with $k > 1$ and $n > 3$, 
there is a planar lattice $L$ with the following two properties\tup{:}
\begin{enumeratei}
\item $L$ is generated by its $n$\! atoms; 
\item $L$ has more than $k$\! elements.
\end{enumeratei}
\end{theorem}

$2$- and $3$-generated AGP lattices 
are enumerated in Section~\ref{S:AGP}.
We prove Theorem~\ref{T:boundary} in 
Section~\ref{S:Boundary}
and Theorem~\ref{T:size} in Section~\ref{S:Large}.

For the basic concept and notation, 
see my books \cite{LTF} and \cite{CFL2}.

As usual, a planar lattice (finite, by definition)
is a lattice with a planar diagram 
understood but not specified.
For a more precise approach, 
see G. Cz\'edli and G.~Gr\"atzer~\cite{CGa}
and K.\,A. Baker and G. Gr\"atzer~\cite{BG18}.

\subsection*{Acknowledgement}
I would like to thank Kirby Baker for his help 
and his contributions to this topic.

\section{Small AGP lattices}\label{S:AGP}

We will utilize the following statement, see 
G. Cz\'edli \cite[3.13A]{gC14a}.

\begin{lemma}[Left-right Inequality]\label{L:lri}
If $L$ is a planar lattice and 
$a \slo b$ and $b \slo c$ in $L$, 
then $b < a \jj c$.
\end{lemma}

If the AGP lattice $L$ has $2$ atoms, 
then $L$ must be isomorphic to $\SC{2}^2 = \SB{2}$. 
The following lemma describes an AGP lattice with $3$ atoms.

\begin{lemma}\label{L:3}
Let $L$ be an AGP lattice with $3$ atoms, 
$a_1, a_2, a_{3}$, enumerated from left to right.
Then $a_1 \jj a_{3}$ is the unit element of $L$.
\end{lemma}

\begin{proof}
Since $L$ is an AGP lattice with $3$ atoms, 
it follows that $1 = a_1 \jj a_2 \jj a_3$.
So we obtain that $1 = a_1 \jj a_{3}$ 
by the Left-right Inequality.
\end{proof}

If $L$ has $3$ atoms, then by Lemma~\ref{L:3},
there are four possibilities, up to isomorphism,
see Figure~\ref{F:4}. 
We see that Theorem~\ref{T:boundary} holds
in all four cases.

\section{Boundary chains}\label{S:Boundary}

To prepare for the proof of Theorem~\ref{T:boundary},
we need some preliminary results.
In this section,
let $L$ be an AGP lattice with its atoms 
$a_1, \dots, a_{n}$ enumerated from left to right.
For $u \in L$, define 
\begin{equation}\label{E:At}
   \At(u) = \setm{i}{a_i \leq u}.
\end{equation}
Let 
\begin{equation}\label{E:Cn}
\SC n = \set{1 \prec 2 \prec \dots \prec n}
\end{equation}
be the $n$ element chain.

\begin{lemma}\label{L:interval}
The set $\At(u)$ is an interval of $\SC n$ 
for any $u \in L$.
Moreover, 
\begin{align}
   \At(x) \ci \At(y) 
      &\text{\q for all $x \leq y \in L$,}\label{E:3}\\
   \At(u \mm v) = \At(u) \ii \At(v) 
      &\text{\q for all $u, v \in L$.}\label{E:4}
\end{align}
\end{lemma}

\begin{proof}
Let $u \in L$ and let $1 \leq i < j \leq n$
with $i, j \in \At(u)$. 
By the Left-right Inequality, 
$k \in \At(u)$ for any $i < k < j$.
Therefore, $\At(u)$ is an interval. 
The other two statements are also obvious.
\end{proof}

\begin{lemma}\label{L:Mk}
For $k \in \set{1,\dots,n}$, let
\begin{equation}\label{E:Mk}
   M_k = \UUm{\fil a_i}{k \leq i \leq n} \uu \set{0}.
\end{equation}
Then $M_k$ is a sublattice of $L$.
\end{lemma}

\begin{proof}
By construction, $M_k$ is closed under joins. 
To show that $M_k$ is closed under meets,
let $x, y \in M_k$. 
Assume that $x \mm y \notin M_k$. 
Since $0 \in M_k$ but $x \mm y \nin M_k$, 
there exists an atom $a_i \in \id{(x \mm y)}$. 
Clearly, $i<k$ since otherwise 
$x \mm y \in \fil a_i \ci M_k$
would contradict that $x \mm y \nin M_k$. 
By definition, we can pick $j_x,j_y \in [k,\dots,n]$
such that $j_x \ci \At(x))$ and $j_y \in \At(y)$. 
Using equation \eqref{E:3} of Lemma~\ref{L:interval}, 
we have that $i \in \At(x \mm y) \ci \At(x)$ 
and $i \in \At(x \mm y) \ci \At(y)$. 
Since $i, j_x \in \At(x)$ and $i < k \le j_x$, 
Lemma~\ref{L:interval} gives that $k \in \At(x)$. 
Similarly, $k \in \At(y)$. 
Equation \eqref{E:4} of Lemma~\ref{L:interval} 
yields that $k \in \At(x) \ii \At(y) = \At(x \mm y)$, 
that is, $x \mm y \in \fil a_k \ci M_k$, 
contradicting the assumption that $x \mm y \in M_k$. 
\end{proof}

Note an easy consequence of Lemma~\ref{L:Mk}. 
\begin{corollary}\label{C:xx}
For $k < k' \in \set{1,\dots,n}$, let
\begin{equation}\label{E:Mkk}
   M_{k,k'} = \UUm{\fil a_i}{k \leq i \leq k'} \uu \set{0}.
\end{equation}
Then $M_{k,k'}$ is a sublattice of $L$. 
\end{corollary}

\begin{proof}
Intersect the sublattice of Lemma~\ref{L:Mk} 
and its left-right reverse.
\end{proof}

\begin{lemma}\label{L:prop2}
Define the principal ideals 
\begin{equation}\label{E:Jk}
J_k = \id(a_1 \jj \cdots \jj a_{k-1})
\end{equation}
for $k = 1, \dots, n$. 
Then $L = J_k \uu M_k$ for $k = 1, \dots, n$.
\end{lemma}

\begin{proof}
Let $S = J_k \uu M_k$. 
Observe that $S = J_k \uu (M_k - \set{0})$.
Using that $J_k$ is a sublattice 
and a down-set, and also that $M_k - \set{0}$
is an up-set,
it is easy to see that $S$ is a sublattice of $L$.
\end{proof}

Now we prove Theorem~\ref{T:boundary}
in the following form.

\begin{theorem}\label{T:Boundary}
Let $L$ be an AGP lattice with all its atoms 
$a_1, \dots, a_{n}$ 
enumerated from left to right.
Define the elements $c_{\tl,1} = a_1$, 
$c_{\tl,2} = a_1 \jj a_{2}$, 
\dots, $c_{\tl,n} = a_1 \jj a_{n}$.
Then the chain $C$\emph{:}
\begin{equation}\label{E:chain}
    0 < c_{\tl,1} \leq c_{\tl,2}\leq \dots \leq c_{\tl,n} = 1
\end{equation}
is a maximal chain in $L$, in fact, 
the left boundary of $L$.
\end{theorem}

\begin{proof}
The inequalities in \eqref{E:chain}
follow from the Left-right Inequality.
Let $L$ be as in the theorem and let us assume that
$C$ is not maximal, that is,  
\begin{equation}\label{E:x}
a_1 \jj \cdots \jj a_{k-1} < x < a_1 \jj \cdots \jj a_k
\end{equation}
for some $x \in L$.
By Lemma~\ref{L:prop2}, 
we can decompose $L$ as $J_k \uu M_k$. 
By \eqref{E:x}, $x \in M_k$ holds, 
and so $x \geq a_j$ for some $j \in \set{k, \dots, n}$.
Therefore, $[1, k] \ci [1, j] \ci \At(x)$,
and so $a_1 \jj \cdots \jj a_k \leq x$, 
contradicting \eqref{E:x}, so $C$ is maximal.

Since all the elements of $L$ 
on or to the right of $C$ form a sublattice 
containing all its atoms $a_1, \dots, a_{n}$ 
generating $L$ (see for instance, 
David Kelly and Ivan Rival~\cite{KR75}),
it follows that $C$ is on the left boundary of $L$.
\end{proof}

Theorem~\ref{T:Boundary} implies Theorem~\ref{T:boundary} 
by left-right symmetry.

\section{Large AGP lattices with $4$ atoms}\label{S:Large}
 
\begin{figure}[!t]
\centerline{\includegraphics[scale = .7]{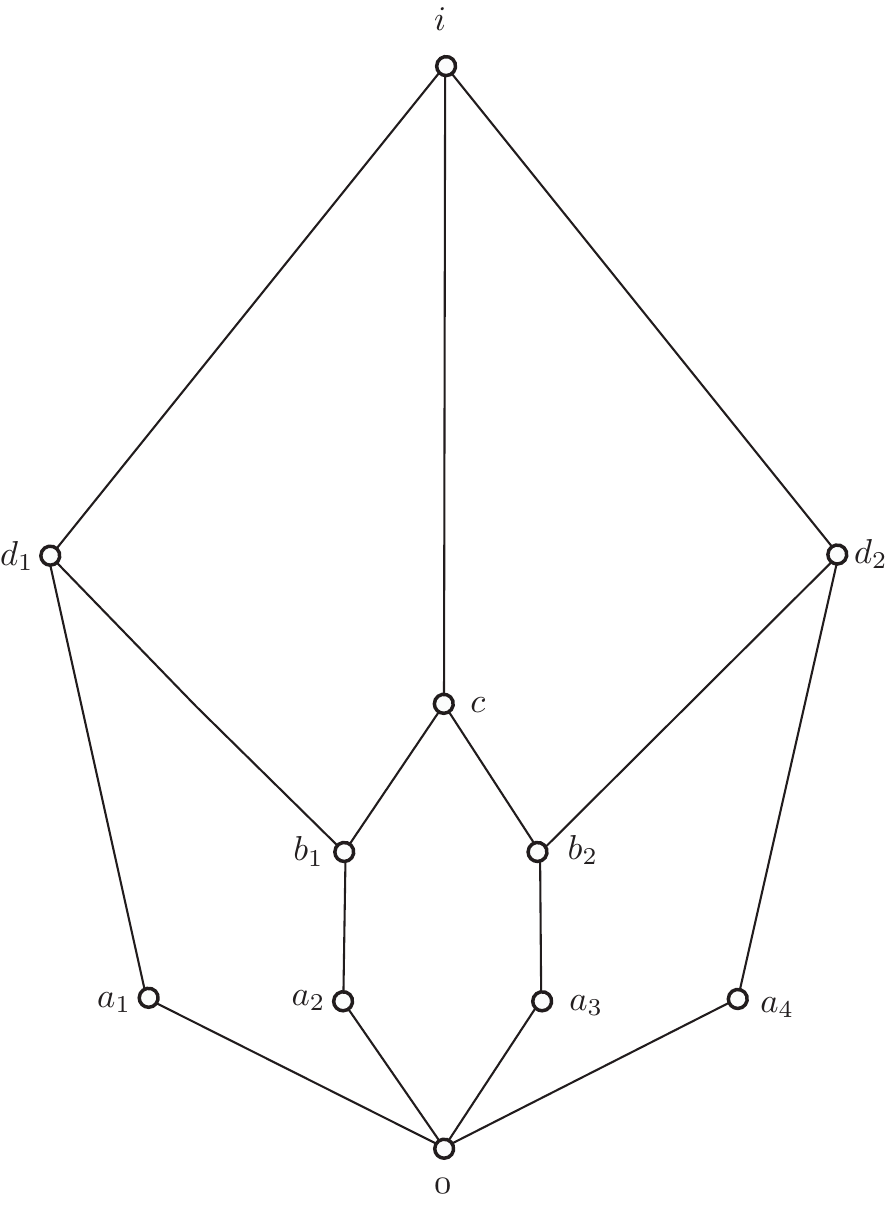}}
\caption{The lattice $\SK 1$}\label{F:4genstep1}
\end{figure}

\begin{figure}[p]
\centerline{\includegraphics[scale = .7]{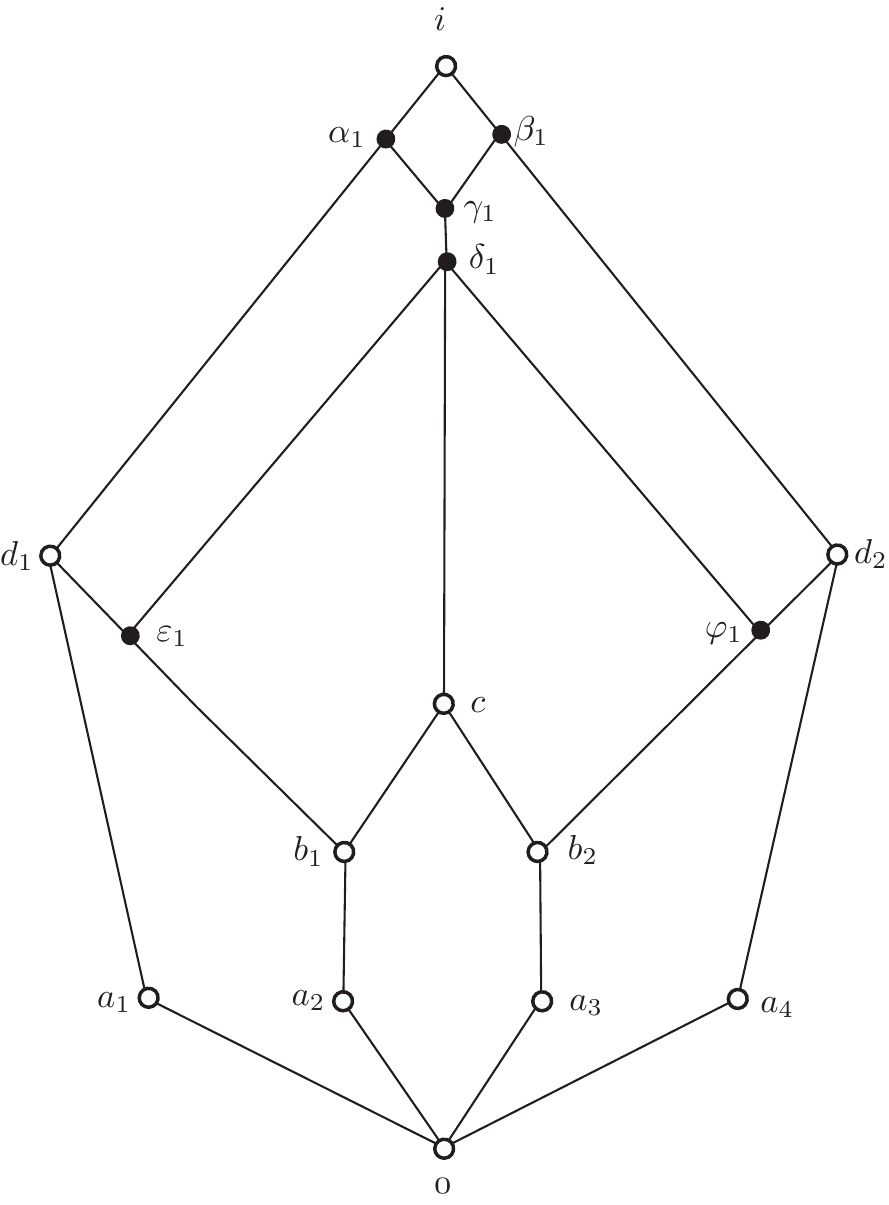}}
\caption{The lattice $\SK 2$}\label{F:4genstep2}

\bigskip

\bigskip

\centerline{\includegraphics[scale = .7]{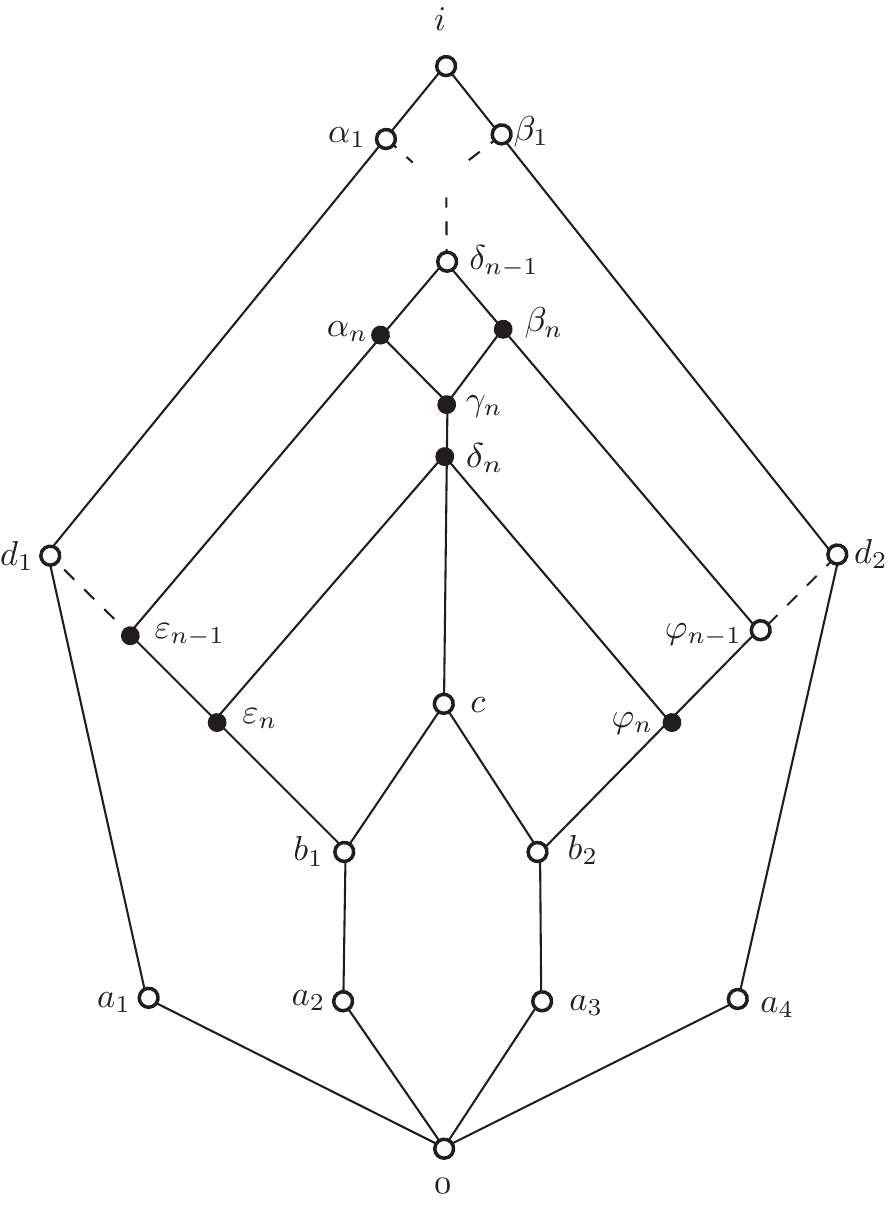}}
\caption{The lattice $\SK n$}\label{F:4geninduct}
\end{figure}

To prove Theorem~\ref{T:size}, we have to construct
arbitrarily large AGP lattices with $4$~atoms.

We start with the lattice $\SK 1$, the ten-element lattice of
Figure~\ref{F:4genstep1}. It is planar.
It~is generated by its atoms; 
in fact, all but two of its elements are joins of atoms.
The exceptions are $b_1$ and $b_2$. 
However, $b_1 = (a_1 \jj a_2) \mm (a_2 \jj a_3)$
and symmetrically for $b_2$.

We construct the lattice $\SK n$ inductively. 
To illustrate the inductive step, 
we first construct the lattice $\SK 2$  
of Figure~\ref{F:4genstep2}.

We obtain $\SK 2$ from $\SK 1$ by adding six elements,
$\ga_1, \gb_1, \gg_1, \gd_1, \ge_1, \gf_1$;
see the black filled elements of Figure~\ref{F:4genstep2}.
Observe that $\SK 1$ is 
a meet-subsemilattice of $\SK 2$ 
and it is almost a join-subsemilattice 
(hence a sublattice) of $\SK 2$. 
The exceptions are $a_1 \jj c$ and $d_1 \jj c$
(they both equal $i$ in $\SK 1$ and they both equal
$\ga_1$ in $\SK 2$) and symmetrically.

For the new elements, we have the following order relations:
\begin{align}\label{E:11}
b_1 &<\ge_1 < d_1, & b_2 &<\gf_1 < d_2, &d_1 &<\ga_1 < i,\\
d_2 &<\gb_1 < i, & \gd_1 &< \gf_1,\notag
\intertext{and the following joins and meets:}
\gg_1 &= \ga_1 \mm \gb_1, & \ga_1 &= d_1 \jj c,
&\gb_1 &= d_2 \jj c, &\gg_1 &= \ga_1 \mm \gb_1\label{E:12}\\
\ge_1 &= d_1 \mm \gg_1 & \gf_1 &= \gg_1 \mm d_1
& 1 &= \ge_1 \mm \gf_1.\notag
\end{align} 

Since $\SK 2$ is planar ordered set with bounds,
it is a lattice.

To see that $\SK 2$ is generated by its atoms, 
first observe that the $4$ changed joins in~$\SK 1$
participate only in verifying 
that $i$ is generated by the atoms,
but, of course, we have $i = a_1 \jj a_2 \jj a_3 \jj a_4$.
So all elements of $\SK 2$ are
generated by the atoms in~$\SK 2$. 

Now the induction step is very similar. 
We start with the lattice $\SK {n-1}$---as 
illustrated by Figure~\ref{F:4geninduct}---and 
extend it to the lattice $\SK n$
by adding the six elements 
$\ga_{n-1}, \gb_{n-1}, \gg_{n-1}, 
\gd_{n-1}, \ge_{n-1}, \gf_{n-1}$,
black filled in Figure~\ref{F:4geninduct}.
With obvious changes in \eqref{E:11} and \eqref{E:12},
we describe $\SK n$ and verify its properties.

\end{document}